\documentclass[12pt,dvips]{article}
\usepackage{amsmath,amsfonts,amssymb,amsthm,graphicx,verbatim,subfig,wasysym}
\usepackage[T1]{fontenc}
\usepackage{a4wide}

\usepackage[all]{xy}

\ifx\pdftexversion\undefined

\input xypic
\xyoption{all}

\newtheorem{theorem}{Theorem}[section]
  \newtheorem{proposition}[theorem]{Proposition}
  \newtheorem{lemma}[theorem]{Lemma}
  \newtheorem{corollary}[theorem]{Corollary}

  \theoremstyle{definition}
  \newtheorem{definition}[theorem]{Definition}
  \newtheorem{example}[theorem]{Example}

\usepackage[a4paper,colorlinks,link
=black,filecolor=black,citecolor=black,urlcolor=black,pdfstartview=FitH]{hyperref}
\else

\usepackage[a4paper,colorlinks,linkcolor=black,filecolor=black,citecolor=black,urlcolor=black,pdfstartview=FitH]{hyperref}
\fi

\font\sixbb=msbm6
\font\eightbb=msbm8
\font\twelvebb=msbm10 scaled 1095
\newfam\bbfam
\textfont\bbfam=\twelvebb \scriptfont\bbfam=\eightbb
                           \scriptscriptfont\bbfam=\sixbb

\newcommand{\Rea}{\mathbb{R}}

\newcommand{\Int}{\mathbb{Z}}

\newcommand{\Nat}{\mathbb{N}}

\newcommand{\KK}{\mathbb{K}}
\newcommand{\norm}[1]{\left\lVert#1\right\rVert}
\newtheorem{theorem}{\bf Theorem}[section]

\newtheorem{proposition}[theorem]{\bf Proposition}
\newtheorem{corollary}[theorem]{\bf Corollary}
\newtheorem{example}[theorem]{\bf Example}
\newtheorem{lemma}[theorem]{\bf Lemma}
\newtheorem{definition}[theorem]{\bf Definition}
\newcommand{\enp}{\begin{flushright} $\Box$ \end{flushright}}
\newcommand{\beq}[0]{\begin{equation}}
\newcommand{\enq}[0]{\end{equation}}
\newcommand{\bbk}{{\bf k}}
\newcommand{\bbm}{{\bf m}}

\newcommand{\bzero}{{\bf 0}}
\newcommand{\bone}{{\bf 1}}
\newcommand{\btwo}{{\bf 2}}
\newcommand{\intd}{\mathring{\Delta}}

\newcommand{\thh}{\tilde{H}}
\newcommand{\mb}[1]{{\mathbf #1}}

\newcommand{\cf}{{\cal F}}

\newcommand{\cg}{{\cal G}}

\newcommand{\ehat}{\widehat{E_{\cf}}}
\newcommand{\ovp}{\overrightarrow{P}}

\newcommand{\bj}{\mb{j}}
\newcommand{\bjp}{\mb{j'}}
\newcommand{\ca}{\mathfrak{A}}
\newcommand{\bfk}{b_{\cf,\bbk}}

\newcommand{\opi}{\overline{\Pi}}
\newcommand{\utt}{{\bf T}}
\newcommand{\ohat}{\widehat{1}}
\newcommand{\mc}[1]{{\mathcal #1}}

\newcommand{\intr}[1]{\mathring{#1}}

\title{Homology of Spaces of Directed Paths \\ in Euclidean Pattern Spaces}
\begin{document}
\author{Roy Meshulam\thanks{Department of Mathematics, Technion, Haifa
    32000, Israel. e-mail: meshulam@math.technion.ac.il~. Supported by
     ISF and GIF grants.}  \and Martin Raussen\thanks{ Department of
    Mathematical Sciences, Aalborg University, Fredrik Bajersvej 7G,
    9220 Aalborg {\O}st, Denmark.  e-mail: raussen@math.aau.dk~. Both
    authors acknowledge support from the ESF research networking
    programme ACAT.}}

\maketitle
\pagestyle{plain}

\vspace{-5mm}
\begin{center}
{\it In memory of Jirka Matou\v{s}ek}
\end{center}
\vspace{3mm}

\begin{abstract}
  Let $\cf$ be a family of subsets of $\{1,\ldots,n\}$ and let
  $$Y_{\cf}=\bigcup_{F \in \cf} \{(x_1,\ldots,x_n) \in \Rea^n: x_i \in
  \Int \text{~for~all~} i \in F\}.$$
  Let $X_{\cf}=\Rea^n \setminus Y_{\cf}$. For a vector of positive
  integers $\bbk=(k_1,\ldots,k_n)$ let
  $\vec{P}(X_{\cf})_{\bzero}^{\bbk+\bone}$ denote the space of
  monotone paths from $\bzero=(0,\ldots,0)$ to
  $\bbk+\bone=(k_1+1,\ldots,k_n+1)$ whose interior is contained in
  $X_{\cf}$. The path spaces $\vec{P}(X_{\cf})_{\bzero}^{\bbk+\bone}$
  appear as natural examples in the study of Dijkstra's PV-model for
  parallel computations in concurrency theory.

  \noindent We study the topology of
  $\vec{P}(X_{\cf})_{\bzero}^{\bbk+\bone}$ by relating it to a
  subspace arrangement in a product of simplices.  This, in
  particular, leads to a computation of the homology of
  $\vec{P}(X_{\cf})_{\bzero}^{\bbk+\bone}$ in terms of certain order
  complexes associated with the hypergraph $\cf$.
\end{abstract}

\section{Introduction}
\label{s:intro}
Concurrency theory in computer systems deals with properties of
systems in which several computations are executing simultaneously and
potentially interacting with each other. Among the many models
suggested for the study of concurrency are the Higher Dimensional
Automata (HDA) introduced by Pratt \cite{Pratt:90}. Those arise as
cubical complexes in which individual cubes (of varying dimension)
with directed paths on each of them, are glued together
consistently. Compared to other concurrency models, HDA have the
highest expressive power based on their ability to represent causal dependence \cite{Glabbeek:06}. On the other hand, only
little is known in general about the topology of the space of directed
paths of a general HDA \cite{Raussen:12a}.

A specific simple case of linear HDA's consists of the PV-model
suggested by Dijkstra \cite{Dijkstra:68} back in the 1960's.  In this
model there are $m$ resources (e.g.~shared memory sites)
$a_1,\ldots,a_m$ with positive integer capacities
$\kappa(a_1),\ldots,\kappa(a_m)$ where $\kappa(a_i)$ indicates the
maximal number of processes that $a_i$ can serve at any given time,
and $n$ linear processes $T_1,\ldots,T_n$ (without branchings or loops) that require access to these
resources. Given a resource $a$ and a process $T$, denote by $Pa$ and
$Va$ the locking and respectively unlocking of $a$ by $T$.  A process
$T_i$ is specified by a sequence of locking and unlocking operations
on the various resources in a certain order.  Modeling each process
$T_i$ as an ordered sequence of integer points on the interval
$(0,k_i]$, one can view a legal execution of $\utt=(T_1,\ldots,T_n)$
as a coordinate-wise non-decreasing continuous path from
$\bzero=(0,\ldots,0)$ to $\bbk+\bone=(k_1+1,\ldots,k_n+1)$ that avoids
a forbidden region determined by the processes and by the capacities
of the resources. If two such paths are homotopic via a homotopy
respecting the monotonicity condition then corresponding concurrent
computations along the two paths have always the same result
(\cite{FGR:06,FGHMR:15}).

Let $X_{\utt,\kappa}$ denote the complement of the
forbidden region in $\prod_{i=1}^n [0,k_i+1]$. The trace space
$\ovp(X_{\utt,\kappa})_{\bzero}^{\bbk+\bone}$ associated with the pair
$(\utt,\kappa)$ consists of all paths as above endowed with the
compact-open topology.  For example, for the two processes sharing two
resources depicted in Figure \ref{figure1}, the forbidden region is
the "Swiss Flag" and the trace space is homotopy equivalent to the two
point space $S^0$.
\begin{figure}
\begin{center}
  {\label{fig:sf}
  \scalebox{0.4}{\input{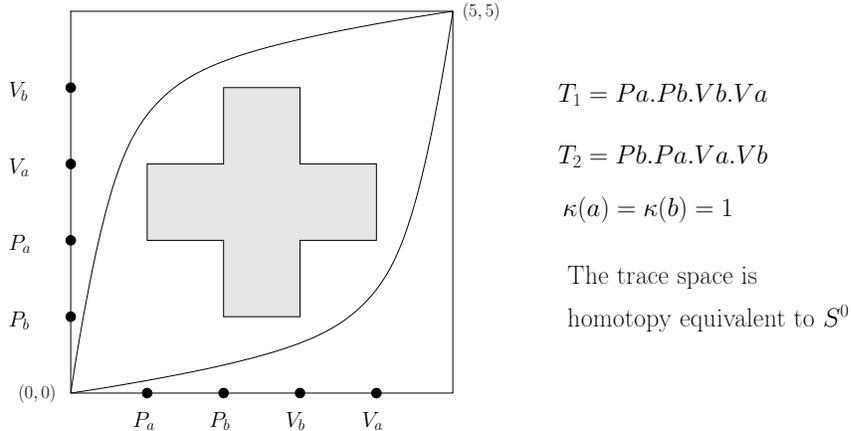}}}
  \caption{The Swiss flag example -- two processes sharing two resources.}
  \label{figure1}
\end{center}
\end{figure}
For an analysis of the PV spaces $X_{\utt,\kappa}$ and their
associated trace spaces $\ovp(X_{\utt,\kappa})_{\bzero}^{\bbk+\bone}$, we refer to
\cite{FGR:06,Raussen:10,Ziemianski:15,FGHMR:15}.

In this paper we consider a special class of PV models in which the
access and release of every resource happen \emph{without time delay}. In
this case, the forbidden region is a union of sets of the form
$B \cap (K_1 \times \cdots \times K_n)$, where $B$ is a fixed aligned
box and each $K_i$ is either $\Int$ or $\Rea$. Our main result (see
Theorem \ref{poin1} below) is a formula for the Poincar\'{e} series of
the trace spaces associated to such special PV models. We proceed with
some formal definitions leading to the statement of Theorem
\ref{poin1}.

Let $X$ be a subspace of $\Rea^n$.  A continuous path
$\mb{p}=(p_1,\ldots ,p_n):I=[0,1] \to X\subset \Rea^n$ is called
\emph{directed} if all components $p_i:I\to\Rea$ are
non-decreasing. For two points $y_0$ and $y_1$ in the closure of $X$,
let $\vec{P}(X)_{\mb{y}_0}^{\mb{y}_1}$ be the space of all directed
paths in $\bar{X}$ (endowed with the compact-open topology) starting at
$\mb{y}_0$ and ending at $\mb{y}_1$ whose interior is contained in
$X$.

Let $\Nat$ denote the non-negative integers and let $\Nat_+$ denote
the positive integers.  Let $\mb{k}=(k_1,\ldots,k_n) \in \Nat_+^n$ be
a fixed vector, and let $\bzero=(0,\ldots,0)$, $\bone=(1,\ldots,1)$,
$\bbk+\bone=(k_1+1,\ldots,k_n+1)$.  In this paper we study the
topology of $\vec{P}(X)_{\bzero}^{\bbk+\bone}$ for spaces $X$ that are
associated with the special PV programs described above. Let $\cf$ be
a family of subsets of $[n]=\{1,\ldots,n\}$ and let
\begin{equation}
\label{d:yf}
Y_{\cf}=\bigcup_{F \in \cf} \{(x_1,\ldots,x_n) \in \Rea^n: x_i \in
\Int \text{~for~all~} i \in F\}.
\end{equation}
The \emph{Euclidean Pattern Space} associated with $\cf$ is defined by
$X_{\cf}=\Rea^n \setminus Y_{\cf}$, with a corresponding \emph{Path Space} $\vec{P}(X_{\cf})_{\mb{0}}^{\mb{k}+\mb{1}}$.
\ \\ \\
{\bf Example:} If $\cf$ consists of the single set $[n]$ then
$X_{\cf}=\Rea^n \setminus\Int^n$.  Raussen and Ziemia\'{n}ski
\cite{RZ14} investigated the path space
$\vec{P}(\Rea^n\setminus \Int^n)_{\mb{0}}^{\mb{k}+\mb{1}}$ and
determined its homology groups and its cohomology ring. Their result
concerning homology is the following:
\begin{theorem}[Raussen and Ziemia\'{n}ski \cite{RZ14}]
For $n \geq 3$
\label{rz}
\begin{equation}
\label{dimh}
\tilde{H}_{\ell}(\vec{P}(\Rea^n\setminus\Int^n)_{\mb{0}}^{\mb{k}+\mb{1}})
= \left\{
\begin{array}{ll}
\Int^{\prod_{i=1}^n \binom{k_i}{m}} & \ell=(n-2)m,~ m>0 \\
0 & \emph{otherwise.}
\end{array}
\right.~~
\end{equation}
\end{theorem}

\noindent The Betti number $\prod_{i=1}^n \binom{k_i}{m}$ in
(\ref{dimh}) corresponds to the number of strictly increasing integer
sequences of length $m$ strictly between $\mb{0}$ and $\mb{k}+\mb{1}$.

In this paper we consider $\vec{P}(X_{\cf})_{\bzero}^{\bbk+\bone}$
for general $\cf$.  Without loss of generality we may assume that
$\cf$ is \emph{upward closed}, i.e. if $F \in \cf$ and
$F \subset F' \subset [n]$ then $F' \in \cf$.
It will also be assumed that $|F| \geq 2$ for all $F \in \cf$ (otherwise
$\vec{P}(X_{\cf})_{\bzero}^{\bbk+\bone}$ is empty).
We first introduce some terminology.
\begin{definition}
\label{d:comb}
$~$
\begin{itemize}
\item[(i)]
A subset $\cg \subset \cf$ is a \emph{matching} if
$G \cap G'=\emptyset$ for all $G \neq G' \in \cg$.  Let $M(\cf)$
denote the family of all nonempty matchings of $\cf$, with partial order $\preceq$ given by $\cg \preceq \cg'$ if for every
$G \in \cg$ there exists a $G' \in \cg'$ such that $G \subset G'$.
For $K \subset [n]$ let
$$M(\cf)_{\preceq K}=\{\cg \in M(\cf): G \subset K \text{~for~all~} G \in \cg\}$$
and let $M(\cf)_{\prec K}=M(\cf)_{\preceq}(K) \setminus \{\{K\}\}$.  The order
complex of $M(\cf)_{\prec K}$ is denoted by
$\Delta(M(\cf)_{\prec K})$.
\item[(ii)]
For a function $\bbm: \cf \rightarrow \Nat$ let $T_{\cf}(\bbm)$ be the
(simple undirected) graph on the vertex set
$\cup_{F \in \cf} \{F\} \times [\bbm(F)]$, where two vertices
$(F,i) \neq (F',i')$ are connected by an edge if
$F \cap F' \neq \emptyset$.\\
\item[(iii)]
An \emph{orientation} of a simple undirected graph $G=(V,E)$ will be
determined by a function $\alpha:E \rightarrow V^2$ that maps an edge
$\{u,v\} \in E$ to either $(u,v)$ or $(v,u)$. An orientation is
\emph{acyclic} if the resulting directed graph does not contain directed
cycles.  Let $\ca(G)$ denote the set of acyclic orientations of $G$
and let $a(G)=|\ca(G)|$. By a result of Stanley \cite{Stanley73}, $a(G)$ can be computed by evaluating the chromatic polynomial of $G$ at $-1$.
\end{itemize}
\end{definition}
\noindent
For $\bbm: \cf \rightarrow \Nat$ let
\begin{equation}
\label{d:bfk}
\begin{split}
  \bfk(\bbm)&=\frac{a(T_{\cf}(\bbm))}{\prod_{F \in \cf} \bbm(F)!}
  \prod_{i=1}^n
  \binom{k_i}{\sum_{F \ni i} \bbm(F)} ~~,\\
  c_{\cf}(\bbm)&=\sum_{F \in \cf} \bbm(F)(|F|-2)+1.
\end{split}
\end{equation}
The \emph{reduced Poincar\'{e} series} of a space $Y$ over a field
$\KK$ is defined by

\begin{equation}
  f_{\KK}(Y,t)=\sum_{i \geq 0} \dim \thh_{i-1}(Y;\KK) t^i.
\label{eq:Poiser}
\end{equation}

\noindent Our main result is the following
\begin{theorem}
\label{poin1}
$~$ \\
(i) If $H_*(\Delta(M(\cf)_{\prec F});\Int)$ is free for all
$F \in \cf$ then $H_*(\vec{P}(X_{\cf})_{\bzero}^{\bbk+\bone};\Int)$ is
free.
\\
(ii) For any field $\KK$
\begin{equation}
\label{po2}
f_{\KK}\left(\vec{P}(X_{\cf})_{\bzero}^{\bbk+\bone},t\right)=\sum_{0 \neq \bbm \in \Nat^{\cf}}  \bfk(\bbm)t^{c_{\cf}(\bbm)}
\prod_{F \in \cf} f_{\KK}\left(\Delta(M(\cf)_{\prec F}),t^{-1}\right)^{\bbm(F)}.
\end{equation}
\end{theorem}

The paper is organized as follows: In Section \ref{s:arrangements} we
describe a subspace arrangement $D_{\cf}$ that is homotopy equivalent
to $\vec{P}(X_{\cf})_{\bzero}^{\bbk+\bone}$.  In Section \ref{s:wedge}
we state Theorem \ref{homgamma1} that describes the homotopy type of
the Alexander dual of $D_{\cf}$ and then use it to prove Theorem
\ref{poin1}. The proof of Theorem \ref{homgamma1} is given in Section
\ref{s:hoehat} which constitutes the main technical part of the paper.
In Section \ref{s:appl} we discuss several applications arising from
particular cases of Theorem \ref{poin1}. The easy conclusions about (higher) connectivity of path spaces in Section \ref{st:conn} are probably 
the most notable ones for applications in concurrency theory. Some open problems are
mentioned in Section \ref{s:remarks}.

\section{Directed Paths via Subspace Arrangements}
\label{s:arrangements}
Spaces of directed paths in a PV-model have been shown to be homotopy
equivalent to certain finite prod-simplicial complexes that make
homology computations possible -- at least in principle
\cite{Raussen:10,FGHMR:15}. Unfortunately, these complexes grow very
fast in dimension and size. Here we give an alternative description as
\emph{complement of a subspace arrangement}. Remark that directedness
has the consequence that such an arrangement has to be considered as a
subset of a product of simplices and \emph{not} of Euclidean space;
this is the reason why classical results are not immediately
applicable.

Let $\mc{F}$ be an upward closed hypergraph on $[n]$ and let
$\mb{k}=(k_1,\ldots,k_n) \in \Nat_+^n$.  In this section we describe a
model for $\vec{P}(X_{\cf})_{\bzero}^{\bbk+\bone}$ up to homotopy
equivalence.

\begin{definition}
\label{d:simplex}
$~$
\begin{itemize}
\item[(i)]
For $k \geq 1$ let $\intd_k$ denote the open $k$-simplex
$$\intd_k =\{(x_1,\ldots,x_k) \in \Rea^k:0 < x_1 < \cdots < x_k < 1\}.$$
For $k=0$ let $\intd_0$ denote the one point space $\{*\}$.\\
For $\mb{k}=(k_1,\ldots,k_n) \in \Nat_+^n$ let
\begin{equation}
  \label{eq:Delta}
  N=\sum_{i=1}^n k_i,\; [\bbk]=\prod_{i=1}^n[k_i], \mbox{ and }
\intd_{\bbk}= \prod_{i=1}^n \intd_{k_i} \subset \Rea^N.
\end{equation}
\item[(ii)] For $F\subset [n]$ let $[\bbk_F]=\prod_{i \in F}
  [k_i]$.
  For $\bj=(\bj(i))_{i\in F} \in [\bbk_F]$ and $F' \subset F$, the
  \emph{restriction} $(\bj)_{|F'} \in [\bbk_{F'}]$ of $\bj$ to $F'$ is
  given by $(\bj)_{|F'}(i)=\bj(i)$ for all $i \in F'$. A \emph{partial
    sequence} is a pair $(F,\bj)$ where $F \subset [n]$ and
  $\bj=(\bj(i))_{i\in F} \in [\bbk_F]$. Let $S_{\cf}$ be the family of
  all partial sequences $(F,\bj)$ where $F \in \cf$ and
  $\bj \in [\bbk_{F}]$.
\item[(iii)]
For a partial sequence
$(F,\bj)$ let
$$G_{(F,\bj)}=
\{(x_{i1},\ldots, x_{ik_i})_{i=1}^n \in \prod_{i=1}^n \intd_{k_i}:
x_{i\bj(i)}=x_{i'\bj(i')} \text{~for~all~} i,i' \in F\}.$$
\\
Let
$$E_{\cf}=\bigcup_{(F,\bj) \in S_{\cf}}
G_{(F,\bj)}~~~~,~~~~D_{\cf}=\intd_{\bbk}-E_{\cf}.$$
\item[(iv)]
The one-point compactification of $\intd_{\bbk}$ is given by
$$\widehat{\intd}_{\bbk}=\intd_{\bbk}
\cup\{\infty\}=\Delta_{\bbk}/_{\partial\Delta_{\bbk}}\cong S^N.$$
For $(F,\bj) \in S_{\cf}$, the compactification
of $G_{(F,\bj)}$ in $\widehat{\intd}_{\bbk}$ is given by
$\Gamma_{(F,\bj)}=G_{(F,\bj)} \cup \{\infty\}$.  The compactification
of $E_{\cf}$ in $\widehat{\intd}_{\bbk}$ is $$\ehat=E_{\cf} \cup \{\infty\}.$$
\end{itemize}
\end{definition}

Let
$\vec{P}_<(X_{\cf})_{\bzero}^{\bbk+\bone}\subset\vec{P}(X_{\cf})_{\bzero}^{\bbk+\bone}$
denote the space of \emph{increasing} directed paths
$\mb{p}=(p_1,\ldots ,p_n):I\to X_{\cf}\subset\mb{R}^n$ characterized
by $t<t'\Rightarrow p_i(t)<p_i(t')$ (instead of $\le$) for all $i$.
Remark that every component $p_i$ is a homeomorphism of the unit
interval.

A correspondence between the space $D_{\mc{F}}$ from Definition
\ref{d:simplex}(iii) and this path space
$\vec{P}_<(X_{\cf})_{\bzero}^{\bbk+\bone}$ and may be established as
follows: For every $k\in\Nat_+$ and
$\mb{x}=(x_1,\dots ,x_k)\in\intr{\Delta}_k$ let
$p_{\mb{x}}:I\to [0,k+1]$ denote the (directed) path with
$p_{\mb{x}}(0)=0$, $p_{\mb{x}}(1)=k+1$,
$p_{\mb{´x}}(x_i)=i, 1\le i\le k,$ and connected by line segments
inbetween. For every $\mb{0}<\mb{k}\in\Nat^n_+$ and every
$\mb{x}=(\mb{x}_1,\dots \mb{x}_n)\in\intr{\Delta}_{\mb{k}}$
(cf.~\ref{eq:Delta}), let
$\mb{p}(\mb{x})(t)=(p_{\mb{x}_1}(t),\dots ,p_{\mb{x}_n}(t))$. This
recipe defines a continuous map
$P:\intr{\Delta}_{\mb{k}}\to\vec{P}_<(\mb{R}^n)_{\mb{0}}^{\mb{k}+\mb{1}}$
that restricts to a map
$P_{\mc{F}}^<:
D_{\mc{F}}\to\vec{P}_<(X_{\mc{F}})_{\mb{0}}^{\mb{k}+\mb{1}}$:
For $\mb{x}=(\mb{x}_1,\dots ,\mb{x}_n)\in\intr{\Delta}_{\mb{k}}$ and
$\mb{x}_i=(x_{i1},\dots ,x_{i k_i})\in\intr{\Delta}_{k_i} ,\;
(F,\mb{j})\in S_{\mc{F}}$
and $0<t<1$ assume that $p_{\mb{x}_i}(t)=\mb{j}(i)\in\mb{Z}, i\in F$.
Then $t=x_{i\mb{j}(i)}=x_{i'\mb{j}(i')}$ for $i,i'\in F$ and hence
$\mb{x}\in E_{\mc{F}}$.

The composition of $P_{\mc{F}}^<$ with the inclusion map
$i: \vec{P}_<(X_{\mc{F}})_{\mb{0}}^{\mb{k}+\mb{1}}
\hookrightarrow\vec{P}(X_{\mc{F}})_{\mb{0}}^{\mb{k}+\mb{1}}$
will be denoted by
$\vec{P}_{\mc{F}}:
D_{\mc{F}}\to\vec{P}(X_{\mc{F}})_{\mb{0}}^{\mb{k}+\mb{1}}$.

\begin{proposition}\label{prop:homeq}
  The map
  $\vec{P}_{\mc{F}}: D_{\mc{F}}\to\vec{P}(X_{\mc{F}})_{\mb{0}}^{\mb{k}+\mb{1}}$
  is a homotopy equivalence.
\end{proposition}

We prove Proposition \ref{prop:homeq} via the following two lemmas:

\begin{lemma}
  The map
  $P_{\mc{F}}^<:
  D_{\mc{F}}\to\vec{P}_<(X_{\mc{F}})_{\mb{0}}^{\mb{k}+\mb{1}}$
  is a homotopy equivalence.
\end{lemma}

\begin{proof}
  Define a reverse continuous map
  $Q:\vec{P}_<(\mb{R}^n)_{\mb{0}}^{\mb{k}+\mb{1}}\to\intr{\Delta}_{\mb{k}}$
  as follows: For
  $\mb{p}=(p_1,\dots ,p_n) \in
  \vec{P}_<(\mb{R}^n)_{\mb{0}}^{\mb{k}+\mb{1}}$
  such that $p_j(x_{ij_{i}})=j_i$ let
  $Q(\mb{p})=(x_{11},\dots ,x_{1k_1}; \dots ; x_{n1},\dots
  ,x_{nk_n})$.
  Remark that $Q$ is well-defined and continuous since every $p_i$ is
  a homeomorphism; and that $Q$ cannot be extended to the space
  $\vec{P}(\mb{R}^n)_{\mb{0}}^{\mb{k}+\mb{1}}$ of non-decreasing
  directed paths. Remark moreover that $Q$ restricts to a map
  $Q_{\mc{F}}:\vec{P}_<(X_{\mc{F}})_{\mb{0}}^{\mb{k}+\mb{1}}\to
  D_{\mc{F}}$.

  It is obvious from the definitions that $Q\circ P$ is the identity
  map on $\intr{\Delta}_{\mb{k}}$ and hence that
  $Q_{\mc{F}}\circ P_{\mc{F}}$ is the identity map on
  $D_{\mc{F}}$. The map
  $P\circ
  Q:\vec{P}_<(\mb{R}^n)_{\mb{0}}^{\mb{k}+\mb{1}}\to\vec{P}_<(\mb{R}^n)_{\mb{0}}^{\mb{k}+\mb{1}}$
  has the property:
  $((P\circ Q)(\mb{p}))_i(p_i^{-1}(j))=j=p_i(p_i^{-1}(j))$ and
  $((P\circ Q)(\mb{p}))_i(t)\not\in\mb{Z}$ for
  $\mb{p}\in\vec{P}_<(\mb{R}^n)_{\mb{0}}^{\mb{k}+\mb{1}}$ and
  $t\not\in Q_i(\mb{p})\cup\{ 0,1\}$. That same property holds for all
  directed paths in the linear homotopy on
  $\vec{P}_<(\mb{R}^n)_{\mb{0}}^{\mb{k}+\mb{1}}$ given by
  $s\mapsto (1-s)\mb{p}+s(P\circ Q)(\mb{p}), 0\le s\le 1$. Hence
  $P\circ Q$ restricts to a map
  $P_{\mc{F}}\circ
  Q_{\mc{F}}:\vec{P}_<(X_{\mc{F}})_{\mb{0}}^{\mb{k}+\mb{1}}\to\vec{P}_<(X_{\mc{F}})_{\mb{0}}^{\mb{k}+\mb{1}}$
  that is homotopic to the identity map on
  $\vec{P}_<(X_{\mc{F}})_{\mb{0}}^{\mb{k}+\mb{1}}$.
\end{proof}

\begin{lemma}\label{lemma:incl}
  The inclusion map
  $i: \vec{P}_<(X_{\mc{F}})_{\mb{0}}^{\mb{k}+\mb{1}}
  \hookrightarrow\vec{P}(X_{\mc{F}})_{\mb{0}}^{\mb{k}+\mb{1}}$
  is a homotopy equivalence for every positive integer vector
  $\mb{k}$.
\end{lemma}

\begin{proof}
  Let $\delta_{\mb{k}}\in\vec{P}_<(\mb{R}^n)_{\mb{0}}^{\mb{k}+\mb{1}}$
  denote the linear path given by
  $\delta_{\mb{k}}(t)=t(\mb{k}+\mb{1})$. Then, for every
  $\mb{p}\in\vec{P}(\mb{R}^n)_{\mb{0}}^{\mb{k}+\mb{1}}$ and
  $0<s\le 1$, the convex combination
  $\mb{p}_s:=(1-s)\mb{p}+s\delta_{\mb{k}}$ is \emph{strictly}
  increasing and hence contained in
  $\vec{P}_<(\mb{R}^n)_{\mb{0}}^{\mb{k}+\mb{1}}$. For a given
  $\mb{p}\in\vec{P}_<(X_{\mc{F}})_{\mb{0}}^{\mb{k}+\mb{1}}$, we want
  to choose $s>0$ small enough to ensure that $\mb{p}_s$ avoids
  $Y_{\mc{F}}$ (see (\ref{d:yf})) and hence so that $\mb{p}_s$ is
  contained in $\vec{P}_<(X_{\mc{F}})_{\mb{0}}^{\mb{k}+\mb{1}}$; and
  this in a way that makes the parameter $s$ depend continuously on
  the path $\mb{p}$.

  Fix a norm and the associated metric $d$ on $\mb{R}^n$, e.g., the
  box norm. For every path
  $\mb{p}\in\vec{P}_<(X_{\mc{F}})_{\mb{0}}^{\mb{k}+\mb{1}}$, the
  spaces $\mb{p}(I)$ and $Y_{\mc{F}}\cap [\mb{0},\mb{k}+\mb{1}]$ are
  disjoint closed and hence compact subspaces of
  $[\mb{0},\mb{k}+\mb{1}]$ with a positive distance
  $d(\mb{p}):=\max_{t\in I}(d(\mb{p}(t),Y_{\mc{F}})$ depending
  continuously on $\mb{p}$. Let $K:=\max_1^nk_i$ and let
  $s(\mb{p})=\frac{d(\mb{p})}{K}$. Then, for every
  $\mb{p}\in\vec{P}_<(X_{\mc{F}})_{\mb{0}}^{\mb{k}+\mb{1}}$ one
  obtains:
  $d(\mb{p},\mb{p}_s)=s(\mb{p})\parallel\mb{p}-\delta_{\mb{k}}\parallel<d(\mb{p})$
  and in particular $d(\mb{p}_s,Y_{\mc{F}})>0$.

  Let
  $i:\vec{P}_<(X_{\mc{F}})_{\mb{0}}^{\mb{k}+\mb{1}}\to\vec{P}(X_{\mc{F}})_{\mb{0}}^{\mb{k}+\mb{1}}$
  denote the inclusion map, and let
  $r:\vec{P}(X_{\mc{F}})_{\mb{0}}^{\mb{k}+\mb{1}}\to\vec{P}_<(X_{\mc{F}})_{\mb{0}}^{\mb{k}+\mb{1}}$
  denote the continuous map given by
  $r(\mb{p})=(1-s(\mb{p}))\mb{p}+s(\mb{p})\delta_{\mb{k}}$. The
  continuous map
  $R:\vec{P}(X_{\mc{F}})_{\mb{0}}^{\mb{k}+\mb{1}}\times I\to
  \vec{P}(X_{\mc{F}})_{\mb{0}}^{\mb{k}+\mb{1}}$
  given by $R(\mb{p},t)=(1-ts(\mb{p}))\mb{p}+ts(\mb{p})\delta_{\mb{k}}$
  is a homotopy between the identity and $i\circ r$; its restriction
  to $\vec{P}_<(X_{\mc{F}})_{\mb{0}}^{\mb{k}+\mb{1}}$ is a homotopy
  between the identity and $r\circ i$.
\end{proof}

\noindent \emph{Remark:}
A variant of the proof above shows that spaces of increasing and of
non-decreasing directed paths (as they arise in models for concurrency
theory) are homotopy equivalent in a more general context.

\section{The Homology of $D_{\cf}$}
\label{s:wedge}
In this section we state Theorem \ref{homgamma1} that describes the
homotopy type of the Alexander dual of $D_{\cf}$ in the one-point
compactification of $\intd_{\bbk}$ -- a sphere of dimension
$N=\sum_{i=1}^n k_i$. This result is then used to prove Theorem
\ref{poin1}.  Our main observation is the following homotopy
decomposition of $\ehat$ (see (\ref{d:bfk}) and Definitions
\ref{d:comb}(i) and \ref{d:simplex}(iii)).
\begin{theorem}
\label{homgamma1}
\begin{equation}
\label{decom}
\ehat \simeq \bigvee_{0 \neq \bbm \in \Nat^{\cf}} \bigvee^{\bfk(\bbm)} S^{N-c_{\cf}(\bbm)} \ast
\underset{F \in \cf}{\varhexstar}  \Delta(M(\cf)_{\prec F})^{*\bbm(F)}.
\end{equation}
\end{theorem}
\noindent
The proof of Theorem \ref{homgamma1} is deferred to Section \ref{s:hoehat}.
\ \\ \\
{\bf Proof of Theorem \ref{poin1}:} (i) If the integral
homology $\thh_*\left(\Delta(M(\cf)_{\prec F});\Int\right)$ is free
for all $F \in \cf$ , then (\ref{decom}) implies that
$\thh_*(\ehat;\Int)$ is free.  Recalling that
$\widehat{\intd}_{\bbk}\cong S^N$, it follows by Alexander duality
that for all $\ell$
\begin{equation*}
\begin{split}
  \thh_{\ell}(D_{\cf};\Int) &=
  \tilde{H}_{\ell}(\intd_{\bbk}-E_{\cf};\Int) \\ &=
  \tilde{H}_{\ell}(\widehat{\intd}_{\bbk}-\ehat;\Int) \cong
  \thh_{N-\ell-1}(\ehat;\Int).
\end{split}
\end{equation*}
Therefore $\thh_{\ell}(D_{\cf};\Int)$ is free.
\\
(ii) Recall that the behavior of the reduced Poincar\'{e} series
$f_{\KK}(\cdot)$ (cf.~(\ref{eq:Poiser})); as a consequence, with
respect to the wedge and join operations is given by
\begin{equation}
\label{wj}
\begin{split}
f_{\KK}(Y_1 \vee Y_2,t)&=f_{\KK}(Y_1,t)+f_{\KK}(Y_2,t)~~,\\
f_{\KK}(Y_1*Y_2,t)&=f_{\KK}(Y_1,t)f_{\KK}(Y_2,t).
\end{split}
\end{equation}
Furthermore, if $Y$ is a subcomplex of $S^N$ then by Alexander duality
\begin{equation}
\label{alexander}
f_{\KK}(S^N-Y,t)=t^{N+1} f_{\KK}(Y,t^{-1}).
\end{equation}
Theorem \ref{homgamma1} together with (\ref{wj}) imply that for any
field $\KK$
\begin{equation}
\label{pehat}
f_{\KK}(\ehat,t)=\sum_{0 \neq \bbm \in \Nat^{\cf}}  \bfk(\bbm)t^{N-c_{\cf}(\bbm)+1}
\prod_{F \in \cf} f_{\KK}(\Delta(M(\cf)_{\prec F}),t)^{\bbm(F)}.
\end{equation}
Combining Proposition \ref{prop:homeq} with (\ref{alexander}) and
(\ref{pehat}) it follows that
\begin{equation*}
\begin{split}
  f_{\KK}\left(\vec{P}(X_{\cf})_{\bzero}^{\bbk+\bone},t\right)
  &=f_{\KK}\left(D_{\cf},t\right)=
  t^{N+1}f_{\KK}\left(\ehat,t^{-1}\right) \\ &=\sum_{0 \neq \bbm \in
    \Nat^{\cf}} \bfk(\bbm)t^{c_{\cf}(\bbm)} \prod_{F \in \cf}
  f_{\KK}\left(\Delta(M(\cf)_{\prec F}),t^{-1}\right)^{\bbm(F)}.
\end{split}
\end{equation*}
{\enp}

\section{Homotopy Decomposition of $\ehat$}
\label{s:hoehat}

In this Section we prove Theorem \ref{homgamma1}. Our basic approach
is to apply the Wedge Lemma of Ziegler and \v{Z}ivaljevi\'{c}
\cite{ZZ93} to the cover $\{\Gamma_{(F,\bj)}:(F,\bj) \in S_{\cf}\}$ of
$\ehat$. The actual proof depends on a number of preliminary
results. For notations, we refer the reader to Definition
\ref{d:simplex}.
\begin{definition}
\label{d:sepfam}
Let $R \subset S_{\cf}$.
\begin{itemize}
\item[(i)]
Let
$$G_R=\bigcap_{(F,\bj) \in R} G_{(F,\bj)} ~~~,~~~
\Gamma_R=\bigcap_{(F,\bj) \in R} \Gamma_{(F,\bj)}=G_R
\cup\{\infty\}.$$
\item[(ii)]
$R$ is
\emph{separated} if $\bj(i) \neq \bj'(i)$ for any
$(F,\bj) \neq (F',\bj') \in R$ and $i \in F \cap F'$.
\end{itemize}
\end{definition}


For separated families $R \subset S_{\cf}$ it is sometimes useful to
represent $G_R$ by a diagram with $n$ rows such that the $i$-th row
contains the coordinates $x_{i1}<\cdots <x_{ik_i}$ of $\intd_{k_i}$,
and such that $x_{ij}, x_{i'j'}$ are connected by a dashed line iff
$x_{ij}=x_{i'j'}$ for all $\mb{x} \in G_R$, i.e. iff there exists an
$(F,\bj) \in R$ such that $i,i' \in F$ and $\bj(i)=j$ and
$\bj(i')=j'$.
\begin{example}
\label{ffig1}
Let $k_1=k_2=k_3=2$ and let $R=\{(F_i,\bj_i)\}_{i=1}^3$ where
$F_1=\{1,2\}$, $F_2=\{2,3\}$, $F_3=\{1,3\}$ and
$(\bj_1(1),\bj_1(2))=(1,1)$, $(\bj_2(2),\bj_2(3))=(2,1)$,
$(\bj_3(1),\bj_3(3))=(2,2)$. The diagram of $G_R$ is depicted in
Figure \ref{figure2}.
\end{example}

\begin{figure}
\begin{center}
  {\label{fig:gr}
  \scalebox{0.6}{\input{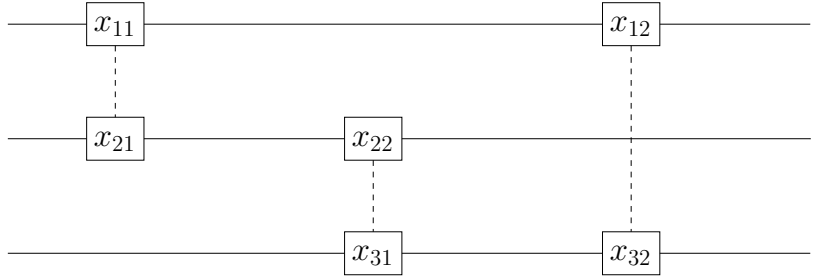}}}
  \caption{Diagram of $G_R$ (cf. Example \ref{ffig1})}
  \label{figure2}
\end{center}
\end{figure}

\begin{definition}
\label{d:kr}
For $R \subset S_{\cf}$
let $K_R$ be the directed graph on the vertex set $R$ with edges
$(F,\bj) \rightarrow (F',\bj')$, where $(F,\bj)$ and $(F',\bj')$ are
distinct elements of $R$ that satisfy $F \cap F' \neq \emptyset$ and
$\bj(i)< \bj'(i)$ for all $i \in F \cap F'$.
The family $R\subset S_{\cf}$ is
\emph{acyclic} if $R$ is separated and if $K_R$ does not contain
directed cycles.  Let $A_{\cf}$ denote the set of all acyclic
subfamilies of $S_{\cf}$.
\end{definition}
\noindent
The next two Propositions describe some properties of $\Gamma_R$ for separated families $R$.
\begin{proposition}
\label{sgam1}
Let $R \subset S_{\cf}$ be a separated family. Then:
\\
(i) If $R \not\in A_{\cf}$  then $\Gamma_R=\{\infty\}$.
\\
(ii) If $R \in A_{\cf}$ then there is a homeomorphism
\begin{equation}
\label{gammar}
\Gamma_R \cong S^{N-\sum_{(F,\bj) \in R} (|F|-1)}.
\end{equation}

\end{proposition}
\noindent
{\bf Proof:} (i) Let
$$(F_1,\bj_1) \rightarrow \cdots \rightarrow (F_r,\bj_r) \rightarrow (F_1,\bj_1)$$
be a directed cycle in $K_R$. Then there exist
$$i_1 \in F_1 \cap F_2, i_2 \in F_2 \cap F_3,\ldots, i_r \in F_r \cap
F_1$$ such that
\begin{equation}
\label{circuit}
\bj_1(i_1)<\bj_2(i_1)~,~ \bj_2(i_2)<\bj_3(i_2)~,~ \cdots~,~
\bj_r(i_r)<\bj_1(i_r).
\end{equation}
We will show that $G_R=\emptyset$ and hence
$\Gamma_R=\{\infty\}$. Indeed, suppose that
$\left((x_{i,1},\ldots, x_{i,k_i})\right)_{i=1}^n$ is contained in
$G_R$. We conclude from (\ref{circuit}) -- since
$i_j,i_{j+1}\in F_{j+1}, j<r$, and $i_1, i_r \in F_1$:
\begin{equation*}
\begin{split}
x_{i_1,\bj_1(i_1)}&<  x_{i_1,\bj_2(i_1)}= x_{i_2,\bj_2(i_2)}
<x_{i_2,\bj_3(i_2)}=x_{i_3,\bj_3(i_3)}  < \\
\cdots &< x_{i_{r-1},\bj_r(i_{r-1})}=x_{i_r,\bj_r(i_r)}< x_{i_r,\bj_1(i_r)}=x_{i_1,\bj_1(i_1)},
\end{split}
\end{equation*}
a contradiction.
\begin{example}
\label{grempty}
Let $k_1=k_2=k_3=2$ and let $R=\{(F_i,\bj_i)\}_{i=1}^3$ where
$F_1=\{1,2\}$, $F_2=\{2,3\}$, $F_3=\{1,3\}$ and
$(\bj_1(1),\bj_1(2))=(2,1)$, $(\bj_2(2),\bj_2(3))=(2,1)$,
$(\bj_3(1),\bj_3(3))=(1,2)$. Then
$(F_1,\bj_1) \rightarrow (F_2,\bj_2) \rightarrow (F_3,\bj_3)
\rightarrow (F_1,\bj_1)$
is a cycle in $K_R$ and thus $G_R=\emptyset$, as is also clear from
the diagram of $G_R$ in Figure \ref{figure3}.
\end{example}

\begin{figure}
\begin{center}
  \scalebox{0.6}{\input{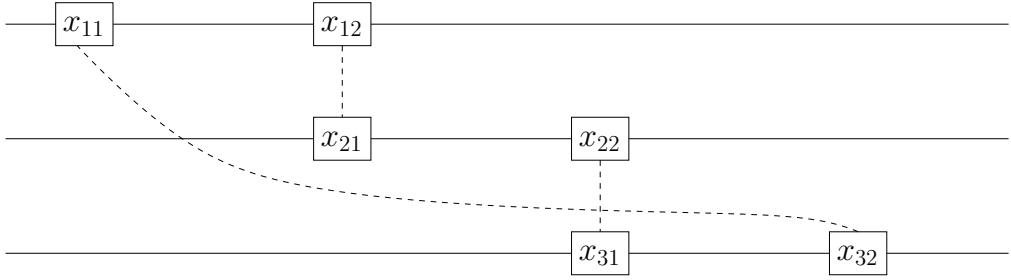}}
  \caption{If $K_R$ has directed cycles then $G_R=\emptyset$ (cf. Example \ref{grempty})}
  \label{figure3}
\end{center}
\end{figure}

\noindent
(ii) For $(F,\bj) \in S_{\cf}$ define
$$
V_{(F,\bj)}= \{(x_{i1},\ldots, x_{ik_i})_{i=1}^n \in \prod_{i=1}^n
\Rea^{k_i}: x_{i\bj(i)}=x_{i'\bj(i')} \text{~for~all~} i,i' \in F\}.
$$
For $R\subset S_{\cf}$, let
$V_R=\bigcap_{(F,\bj) \in R} V_{(F,\bj)}\subset \Rea^N$.
If the family $R$ is separated, then $V_R$ is a linear subspace of
codimension $\sum_{(F,\bj) \in R} (|F|-1)$. If $R$ is acyclic, one can
easily find an element $\mb{x}\in V_R\cap\intr{\Delta}_{\mb{k}}$.

For such a chosen solution $\mb{x}$ and for
$\mb{v}\in S(V_R)=\{\mb{u}\in V_R : \norm{\mb{u}}
=1\}\subset D(V_R)=\{\mb{u}\in V_R: \norm{\mb{u}}\le 1\}$
let
$\alpha (\mb{v})=\min\{ t>0|\
\mb{x}+t\mb{v}\in\partial\Delta_{\mb{k}}\}$;
hence, for $\mb{w}\in\partial\Delta_{\mb{k}}$, one has that
$\alpha(\frac{\mb{w}-\mb{x}}{\norm{\mb{w}-\mb{x}}})=\norm{\mb{w}-\mb{x}}$.
This recipe defines a continuous map $\alpha$ from $S(V_R)$ to the
positive reals since $\alpha (\mb{v})$ is locally obtained as the
minimum among the solutions to a number of linear equations.

We define a (scaling) map
$$\Phi_R:
\Gamma_R=V_R\cap\Delta_{\mb{k}}/_{V_R\cap\partial\Delta_{\mb{k}}}\to
D(V_R)/_{S(V_R)}\cong S^{N-\sum_{(F,\bj) \in R} (|F|-1)}$$
by $$\Phi_R(\mb{w})=
\begin{cases}
  \frac{1}{\alpha(\frac{\mb{w}-\mb{x}}{\norm{\mb{w}-\mb{x}}})}(\mb{w}-\mb{x}) & \mb{w}\neq\mb{x},\\
\mb{0} & \mb{w}=\mb{x}.
\end{cases}$$
The map $\Phi_R$ is indeed a \emph{homeomorphism} with inverse $\Psi_R: D(V_R)/_{S(V_R)}\to\Gamma_R$ given by $$\Psi_R(\mb{v})=
\begin{cases}
  \mb{x}+\alpha(\frac{\mb{v}}{\norm{\mb{v}}})\mb{v} & \mb{v}\neq\mb{0}, \\
  \mb{x} & \mb{v}=\mb{0}.
\end{cases}$$
{\enp}
\begin{proposition}
\label{p:gr}
  Let $R,R' \in A_{\cf}$. Then the following two conditions are
  equivalent:
  \\
  (a) $\Gamma_{R} \subset \Gamma_{R'}$.
  \\
  (b) For any $(F',\bjp) \in R'$ there exists an $(F,\bj) \in R$ such
  that $F' \subset F$ and $\bjp=(\bj)_{|F'}$.
\end{proposition}
\noindent
{\bf Proof:} Clearly (b) implies (a). To show the other direction,
assume that $\Gamma_R \subset \Gamma_{R'}$ and let $(F',\bjp) \in R'$.
Let $R_1=R \cup \{(F',\bjp)\}$, then
$$\Gamma_R=\Gamma_R \cap \Gamma_{R'} \subset \Gamma_{R_1} \subset \Gamma_R,$$
hence $\Gamma_{R_1}=\Gamma_R \neq \{\infty\}$. It follows that if
$R_1$ is separated then it must be acyclic. But this would imply,
using Proposition \ref{sgam1}(ii), that
$$\dim \Gamma_{R_1}= \dim \Gamma_R-(|F'|-1)< \dim \Gamma_R,$$
in contradiction with $\Gamma_{R_1}=\Gamma_R$. Hence $R_1$ is not
separated and therefore
$$S= \{(F,\bj) \in R: F \cap F' \neq \emptyset ~~~ \& ~~~ (\bj)_{|F \cap F'}=(\bjp)_{|F \cap F'}\} \neq \emptyset.$$
We claim that $|S|=1$. Otherwise there exist
$(F_1,\mb{j_1})\neq (F_2,\mb{j_2}) \in R$ and $i_1 \in F_1 \cap F'$,
$i_2 \in F_2 \cap F'$ such that
$\mb{j_1}(i_1)=\mb{j'}(i_1)$ and $\mb{j_2}(i_2)=\mb{j'}(i_2)$. It follows that if
$\mb{x}=(x_{i1},\ldots, x_{ik_i})_{i=1}^n \in G_{R_1}$ then for all $i_1' \in F_1$, $i_2' \in F_2$
\begin{equation}
\label{longeq}
\begin{split}
x_{i_1' \mb{j_1}(i_1')}&=x_{i_1 \mb{j_1}(i_1)} =x_{i_1 \mb{j'}(i_1)} \\
&=x_{i_2 \mb{j'}(i_2)}=x_{i_2 \mb{j_2}(i_2)}=x_{i_2' \mb{j_2}(i_2')}.
\end{split}
\end{equation}
Since $R$ is separated, (\ref{longeq}) implies that $F_1 \cap F_2=\emptyset$.
Let $F_3=F_1 \cup F_2$ and let $\mb{j_3} \in [\bbk_{F_3}]$ be given by
$$
\mb{j_3}(i)= \left\{
\begin{array}{ll}
\mb{j_1}(i) & i \in F_1, \\
\mb{j_2}(i) & i \in F_2.
\end{array}
\right.~~
$$
Writing
$$R_2=R \setminus \{(F_1,\mb{j_1}),(F_2,\mb{j_2})\} \cup
\{(F_3,\mb{j_3})\},$$
it follows from (\ref{longeq}) that
$\Gamma_R=\Gamma_{R_1}\subset \Gamma_{R_2} \subset \Gamma_R$.
Therefore $\Gamma_{R_2}=\Gamma_R \neq \{\infty\}$. As $R_2$ is
separated, it follows that $R_2 \in A_{\cf}$ and hence by Proposition \ref{sgam1}(ii):
$$\dim \Gamma_{R_2}= \dim \Gamma_R + (|F_1|-1)+(|F_2|-1)-(|F_1 \cup F_2|-1)=\dim \Gamma_R -1,$$
in contradiction with $\Gamma_{R_2}=\Gamma_R$. Therefore $|S|=1$.

Write $S=\{(F_1,\mb{j_1})\}$ and let $i_1 \in F_1 \cap F'$. Then
$\mb{j_1}(i_1)=\mb{j'}(i_1)$. It follows that if
$\mb{x}=(x_{i1},\ldots, x_{ik_i})_{i=1}^n \in G_{R_1}$ then for all
$i_1' \in F_1$, $i' \in F'$
\begin{equation}
\label{srteq}
x_{i_1' \mb{j_1}(i_1')}=x_{i_1 \mb{j_1}(i_1)} =x_{i_1 \mb{j'}(i_1)}=x_{i' \mb{j'}(i')}.
\end{equation}
Let $F_4=F_1 \cup F'$ and let $\mb{j_4} \in [\bbk_{F_4}]$ be given by
$$
\mb{j_4}(i)= \left\{
\begin{array}{ll}
\mb{j_1}(i) & i \in F_1, \\
\mb{j'}(i) & i \in F'.
\end{array}
\right.~~
$$
Note that $\mb{j_4}$ is well defined by (\ref{srteq}).
Writing $$R_3=R \setminus\{(F_1,\mb{j_1})\} \cup \{(F_4,\mb{j_4})\},$$
it follows from (\ref{srteq}) that
$$\Gamma_R = \Gamma_{R_1} \subset \Gamma_{R_3} \subset \Gamma_R,$$
hence $\Gamma_{R_3}=\Gamma_R \neq \{\infty\}$.
Furthermore, $|S|=1$ implies that $R_3$ is separated.  Therefore by Proposition \ref{sgam1}(ii):
$$\dim \Gamma_{R_3}=\dim \Gamma_R+(|F_1|-1)-(|F_4|-1).$$
It follows that $|F_1|=|F_4|=|F_1 \cup F'|$. Hence $F' \subset F_1$ and $\mb{j'}=(\mb{j_1})_{|F'}$.
{\enp}

Let $Q$ be the intersection poset of the cover
$\{\Gamma_{(F,\bj)}:(F,\bj) \in S_{\cf}\}$ of $\ehat$ ordered by
reverse inclusion: An element $q$ of $Q$ corresponds to an
intersection $U_q$ of sets in the cover, i.e. $U_q=\Gamma_R$ for
$R \subset S_{\cf}$, and $q' \leq q$ in $Q$ iff $U_q \subset U_{q'}$.
$Q$ has a maximal element $\ohat$ that corresponds to
$U_{\ohat}=\{\infty\}$.  Fix a $\ohat \neq q \in Q$ and let
$R \subset S_{\cf}$ be a family of minimal cardinality such that
$U_q=\Gamma_R$. The assumption that $\cf$ is upward closed implies
that $R$ is a separated family. Indeed, suppose
$u'=(F',\bj') \neq u''=(F'',\bj'') \in R$ and there exists some
$i_0 \in F' \cap F''$ such that $\bj'(i_0)=\bj''(i_0)$. Let
$\infty \neq \mb{x}=(x_{i1},\ldots, x_{ik_i})_{i=1}^n \in
\Gamma_R$. Then for any $i \in F' \cap F''$
$$
x_{i \bj'(i)}=x_{i_0 \bj'(i_0)}=x_{i_0 \bj''(i_0)}=x_{i \bj''(i)},
$$
hence $\bj'(i)=\bj''(i)$. Let $u=(F,\bj) \in S_{\cf}$ where
$F=F' \cup F''$ and
$$\bj(i)
= \left\{
\begin{array}{ll}
\bj'(i) & i \in F', \\
\bj''(i) & i \in F''.
\end{array}
\right.~~
$$
Then $\Gamma_R=\Gamma_{R-\{u',u''\} \cup \{u\}}$, contradicting the
minimality of $R$. Thus $R$ is separated. By Proposition
\ref{sgam1}(i), the assumption $q \neq \ohat$ implies that
$R \in A_{\cf}$; ~cf.~Definition \ref{d:kr}.\\
We next study the topology of the order complex $\Delta(Q_{<q})$.
\begin{proposition}
\label{p:dqq}
Fix $\ohat \neq q \in Q$ and write $U_q=\Gamma_R$ where $R=\{(F_{\ell},\bj_{\ell})\}_{\ell=1}^r \in A_{\cf}$. Then
there is a homeomorphism
\begin{equation}
\label{iso1}
\Delta(Q_{<q}) \cong \Delta(M(\cf)_{\prec F_1})* \cdots * \Delta(M(\cf)_{\prec F_r}) * S^{r-2}.
\end{equation}
\end{proposition}
\noindent {\bf Proof:} Let $M(\cf)_{\preceq F_{\ell}}^*$ denote
the poset obtained by appending to $M(\cf)_{\preceq F_{\ell}}$ a minimal
element $0_{\ell}$. Denote
$$C_q= M(\cf)_{\preceq F_1}^* \times \cdots \times M(\cf)_{\preceq F_r}^* \setminus \{(0_1,\ldots,0_r)\}.$$
Define a mapping
$\gamma: C_q \rightarrow A_{\cf}$
as follows. Let
$\alpha=(\alpha_1,\ldots,\alpha_r) \in C_q$ and
let $$L(\alpha)=\{1 \leq \ell \leq r: \alpha_{\ell} \neq 0_{\ell}\}.$$
Note that $L(\alpha) \neq \emptyset$. For $\ell \in L(\alpha)$ write
$\alpha_{\ell}= \cg_{\ell} \in M(\cf)$ and let
$$\gamma(\alpha)=\bigcup_{\ell \in L(\alpha)} \{(F,(\bj_{\ell})_{|F}): F \in \cg_{\ell}\}.$$
Define an order preserving map $\theta: C_q \rightarrow Q_{\leq q}$ as
follows: For $\alpha \in C_q$ let $\theta(\alpha)$ be the element of
$Q$ that satisfies $U_{\theta(\alpha)}=\Gamma_{\gamma(\alpha)}$.
\begin{lemma}
\label{isop}
$\theta$ is a poset isomorphism.
\end{lemma}
\noindent {\bf Proof:} To show surjectivity, let $q' \leq q$ with
$U_{q'}=\Gamma_{R'}$ for some $R' \in A_{\cf}$.  Then
$\Gamma_R=U_q \subset U_{q'}= \Gamma_{R'}$ and hence, by Proposition
\ref{p:gr}, there exists an $\alpha \in C_q$ such that
$\gamma(\alpha)=R'$. Therefore
$U_{\theta(\alpha)}=\Gamma_{\gamma(\alpha)}=\Gamma_{R'}=U_{q'}$ and so
$\theta(\alpha)=q'$. To show injectivity, assume that
$\theta(\alpha)=\theta(\alpha')$ for some $\alpha,\alpha' \in
C_q$.
Then
$\Gamma_{\gamma(\alpha)}=U_{\theta(\alpha)}=U_{\theta(\alpha')}=\Gamma_{\gamma(\alpha')}$
and therefore $\gamma(\alpha)=\gamma(\alpha')$ by Proposition
\ref{p:gr}. As $\gamma$ is clearly injective, it follows that
$\alpha=\alpha'$.  {\enp}
\noindent
Recall the following result of Walker (Theorem 5.1(d) in \cite{walker88}).
\begin{theorem}[Walker \cite{walker88}]
\label{t:walker}
For $1 \leq i \leq r$ let $T_i$ be a finite poset with minimal element
$0_i$ and maximal element $1_i$.  Let $T=T_1\times \cdots \times T_r$
and $\bzero=(0_1,\ldots,0_r)$, $\bone=(1_1,\ldots,1_r) \in T$.  Let
$\widehat{T_i}=T_i-\{(0_i,1_i)\}$ and
$\widehat{T}=T-\{\bzero,\bone\}$.  Then there is a homeomorphism
$$\Delta(\widehat{T}) \cong \Delta(\widehat{T_1}) * \cdots * \Delta(\widehat{T_r})*S^{r-2}.$$
\end{theorem}
\ \\ \\
For $1 \leq i \leq r$ let $T_i=M(\cf)_{\preceq F_i}^*$. Then
$\widehat{T_i}=M(\cf)_{\prec F_i}$ and
$\widehat{T}=C_q-\{(\{F_1\},\ldots,\{F_r\})\}$. Lemma \ref{isop} thus
implies that $\widehat{T} \cong Q_{<q}$.  Therefore by Theorem
\ref{t:walker}:
\begin{equation*}
\begin{split}
  \Delta(Q_{<q}) &\cong \Delta(\widehat{T}) \cong \Delta(\widehat{T_1}) * \cdots * \Delta(\widehat{T_r})*S^{r-2} \\
  &=\Delta(M(\cf)_{\prec F_1})* \cdots * \Delta(M(\cf)_{\prec F_r}) *
  S^{r-2}.
\end{split}
\end{equation*}
{\enp}

\begin{definition}
\label{d:afm}
For a function $0 \neq \bbm \in \Nat^{\cf}$, let
$A_{\cf}(\bbm)$ denote the set of all $R \in A_{\cf}$ such that
$|\{\bj\in [\mb{k}_F]|\; (F,\bj)\in R\}|=\bbm(F)$ for all $F \in
\cf$.
\end{definition}
\noindent
The final ingredient needed for the proof of Theorem \ref{homgamma1}
is the following computation (see (\ref{d:bfk}) and Definition
\ref{d:comb}(i)).
\begin{proposition}
\label{affm}
Let $0 \neq \bbm \in \Nat^{\cf}$. Then:
\begin{equation}
\label{afmeq}
|A_{\cf}(\bbm)|=
\bfk(\bbm)=\frac{a(T{_{\cf}}(\bbm))}{\prod_{F \in \cf} \bbm(F)!} \prod_{j=1}^n \binom{k_j}{\sum_{F \ni j} \bbm(F)}.
\end{equation}
\end{proposition}
\noindent {\bf Proof:} Let $\tilde{\ca}(T{_{\cf}}(\bbm))$ denote the
set of all acyclic orientations of $T{_{\cf}}(\bbm)$ such that
$(F,i) \rightarrow (F,i')$ for all $F \in \cf$ and
$1 \leq i < i' \leq \bbm(F)$. Then
$$
|\tilde{\ca}(T{_{\cf}}(\bbm))|=\frac{a(T{_{\cf}}(\bbm))}{\prod_{F \in \cf} \bbm(F)!}.
$$
\noindent
Define a mapping
$$\tau: A_{\cf}(\bbm) \rightarrow \tilde{\ca}(T{_{\cf}}(\bbm)) \times \prod_{i=1}^n \binom{[k_i]}{\sum_{F \ni i} \bbm(F)}$$
as follows. Let $R \in A_{\cf}(\bbm)$. For $1 \leq i \leq n$ let
$$B_i=\{\bj(i): (F,\bj) \in R \text{~and~} i \in F\} \in  \binom{[k_i]}{\sum_{F \ni i} \bbm(F)}.$$
Write
$$R=\bigcup_{\{F \in \cf:\bbm(F)>0\}} \{(F,\bj_{F,\ell}): 1\leq \ell \leq \bbm(F)\}$$
where $\bj_{F,\ell} \in \bbk_{F}$ for all $1 \leq \ell \leq \bbm(F)$
and
$$
\bj_{F,1}(i)< \cdots < \bj_{F,\bbm(F)}(i)
$$
for all $i \in F$.
Define an orientation $\alpha \in \tilde{\ca}(T{_{\cf}}(\bbm))$ as follows.
Let $e=\{(F,s),(F',s')\}$ be an edge of $T{_{\cf}}(\bbm)$.
Define $\alpha(e)=\left((F,s),(F',s')\right)$ if either
$F=F'$ and $s<s'$, or if $F\neq F'$ and $\bj_{F,s}(i) < \bj_{F',s'}(i)$
for some (and therefore all) $i \in F \cap F'$.
Now let
$$
\tau(R)=(\alpha,B_1,\ldots,B_n).$$
It is straightforward to check that $\tau$ is bijective. This proves
Proposition \ref{affm}.
\enp

\begin{example}
\label{biject}
To illustrate the bijection $\tau$ from the proof of Claim \ref{affm}
consider the family $\cf=\{F \subset [4]: |F| \geq 2\}$ and let
$n=4$, $(k_1,k_2,k_3,k_4)=(4,5,4,2)$.  Let $F_1=\{1,2\}$,
$F_2=\{2,3\}$ and $F_3=\{1,3,4\}$ and for $F \in \cf$ let
$$
\bbm(F)= \left\{
\begin{array}{ll}
2 & F=F_1~~\emph{or}~~F=F_2, \\
1 & F=F_3, \\
0 & \emph{otherwise.}
\end{array}
\right.~~
$$
Let $R \in A_{\cf}(\bbm)$ satisfy
$\tau(R)=(\alpha,B_1,B_2,B_3,B_4)$ where
$$(B_1,B_2,B_3,B_4)=(\{2,3,4\},\{1,2,4,5\},\{1,3,4\},\{2\})$$
and the orientation $\alpha$ on the (complete) graph $T_{\cf}(\bbm)$
is given by the total order
$$(F_2,1) \rightarrow (F_1,1) \rightarrow (F_1,2) \rightarrow (F_3,1) \rightarrow (F_2,2).$$
The reconstruction of $R$ from $\tau(R)$ is depicted in Figure
\ref{figure5}.
\end{example}

\begin{figure}
\begin{center}
   \scalebox{0.6}{\input{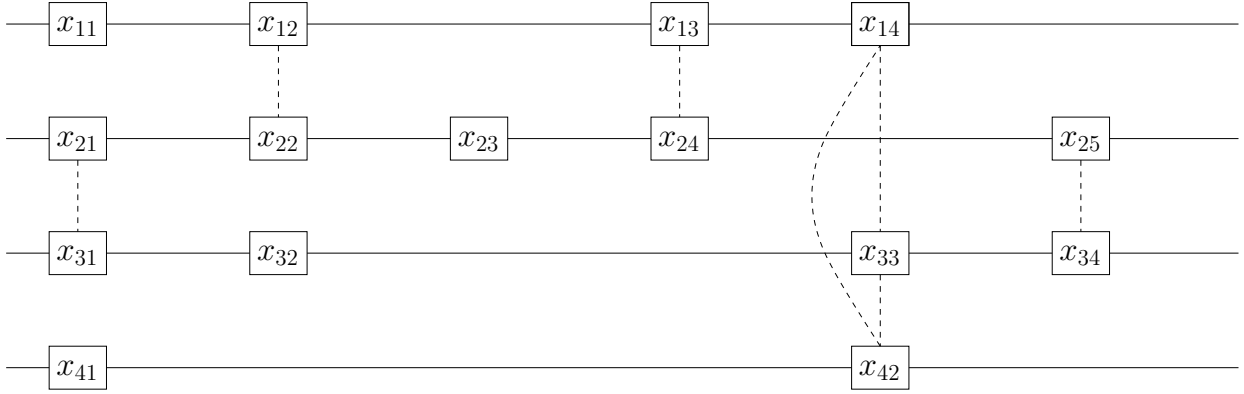}}
  \caption{Reconstruction of $R$  from $\tau(R)$ (cf. Example \ref{biject})}
  \label{figure5}
\end{center}
\end{figure}

\noindent {\bf Proof of Theorem \ref{homgamma1}:} Consider the cover
$\{\Gamma_{(F,\bj)}:(F,\bj) \in S_{\cf}\}$ of $\ehat$, and its
associated intersection poset $Q$ as above.  By Proposition
\ref{sgam1}(ii), if $\ohat \neq q \in Q$ then $U_q$ is a sphere pointed
at $\infty$. Furthermore, if $q<p \in Q$ then the injection
$U_p \rightarrow U_q$ is a pointed embedding of a sphere (or the point
$\infty$ if $p=\ohat$) into a higher dimensional sphere and is thus
homotopic to the constant map $U_p \rightarrow \infty \in
U_q$.
Therefore by the Wedge Lemma (Lemma 1.8 in \cite{ZZ93}) there is a
homotopy equivalence
\begin{equation}
\label{wedge}
\ehat \simeq \bigvee_{q \in Q} \Delta(Q_{<q})*U_q.
\end{equation}
We next determine the contribution of each $q \in Q$ to
(\ref{wedge}). If $q=\ohat$ then $\Delta(Q_{<q})*U_q$ is contractible
to the point $\infty$ and hence does not contribute to
(\ref{wedge}). Suppose $q < \ohat$ and let $U_q=\Gamma_R$ where
$R=\{(F_{\ell},\bj_{\ell})\}_{\ell=1}^r \in A_{\cf}$.  Combining
Proposition \ref{sgam1}(ii) and (\ref{iso1}) it follows that
\begin{equation}
\label{eq:q1}
\begin{split}
  \Delta(Q_{<q})*U_q  &\cong  \Delta(M(\cf)_{\prec F_1})* \cdots * \Delta(M(\cf)_{\prec F_r})*S^{r-2}*S^{N-\sum_{(F,\bj) \in R} (|F|-1)} \nonumber \\
  &\cong S^{N-\sum_{i=1}^r (|F_i|-2)-1} \ast
  \overset{r}{\underset{i=1}{\varhexstar}} \Delta(M(\cf)_{\prec F_i}).
\end{split}
\end{equation}
\noindent
Therefore, if $R \in A_{\cf}(\bbm)$ and $U_q=\Gamma_R$ then
(cf.~(\ref{d:bfk}))
\begin{equation}
\label{gammaz}
\Delta(Q_{<q})* U_q \cong S^{N-c_{\cf}(\bbm)}\ast
\underset{F \in \cf}{\varhexstar}  \Delta(M(\cf)_{\prec F})^{*\bbm(F)}.
\end{equation}
Theorem \ref{homgamma1} now follows from (\ref{wedge}), (\ref{eq:q1}) and
Proposition \ref{affm}.  {\enp}

\section{Applications}
\label{s:appl}

In this section we use Theorem \ref{poin1} to study several specific
Euclidean pattern spaces.

\subsection{The Homology of
  $\vec{P}(\Rea^n\setminus \Int^n)_{\mb{0}}^{\mb{k}+\mb{1}}$}

As noted earlier, $\Rea^n \setminus\Int^n=X_{\cf}$ where $\cf$
consists of the single set $[n]$.  Since $\Delta(M(\cf)_{\prec [n]})$
is the empty complex $\{\emptyset\}$ it follows that
$f_{\KK}(\Delta(M(\cf)_{\prec [n]}),t) =1$. If $\bbm([n])=m>0$ then
$\bfk(\bbm)=\prod_{i=1}^n \binom{k_i}{m}$ and
$c_{\cf}(\bbm)=m(n-2)+1$. Theorem \ref{poin1} implies that
$$f_{\KK}(\vec{P}(\Rea^n\setminus \Int^n)_{\mb{0}}^{\mb{k}+\mb{1}},t)=\sum_{m \geq 1}
\prod_{i=1}^n \binom{k_i}{m} t^{m(n-2)+1}.$$
Since
$\tilde{H}_*(\Delta(M(\cf)_{\prec
  [n]}))=\tilde{H}_{-1}(\{\emptyset\})=\Int$
is free, it follows that
$\tilde{H}_{\ell}(\vec{P}(\Rea^n\setminus\Int^n)_{\mb{0}}^{\mb{k}+\mb{1}})$
is free of rank $\prod_{i=1}^n \binom{k_i}{m}$ if $\ell=(n-2)m>0$, and
is zero otherwise.  This recovers the above mentioned Theorem \ref{rz}
of Raussen and Ziemia\'{n}ski \cite{RZ14}.

\subsection{Binary Path Spaces}
The \emph{binary path space} associated with an upward closed
$\cf \subset 2^{[n]}$ is $\vec{P}(X_{\cf})_{\bzero}^{\btwo}$ where
$\btwo=(2,\ldots,2)$. Note that $\vec{P}(X_{\cf})_{\bzero}^{\btwo}$ is
homotopy equivalent to the diagonal subspace arrangement
$$\Rea^n- \bigcup_{F=\{i_1,\ldots,i_{\ell}\} \in \cf} \{x=(x_1,\ldots,x_n) \in \Rea^n: x_{i_1}=\cdots=x_{i_{\ell}}\}.$$
The general formula (\ref{po2}) for the Poincar\'{e} series of
$\vec{P}(X_{\cf})_{\bzero}^{\bbk+\bone}$ simplifies in this case as
follows.  Let $\mb{k}=\bone$ and let $0 \neq \bbm \in \Nat^{\cf}$.
Then $b_{\cf,\bone}(\mb{m})=1$ if both $\mb{m}(F) \leq 1$ for all
$F \in \cf$, and $\{F: \mb{m}(F)=1\} \in M(\cf)$. Otherwise
$b_{\cf,\bone}(\mb{m})=0$.  Hence, by (\ref{po2})
\begin{equation}
\label{po3}
f_{\KK}\left(\vec{P}(X_{\cf})_{\bzero}^{\btwo},t\right)=\sum_{\cg \in M(\cf)}
t^{\sum_{F \in \cg}(|F|-2)+1} \prod_{F \in \cg}
f_{\KK}\left(\Delta(M(\cf)_{\prec F}),t^{-1}\right).
\end{equation}
Equation (\ref{po3}) can also be obtained from the general
Goresky-MacPherson formula for the homology of subspace arrangements
\cite{GM88}.

\subsection{The $(s,\bbk)$-Equal Path Space}
Let $1 \leq s \leq n$ and $\mb{k}=(k_1,\ldots,k_n) \in \Nat_+^n$.  The
\emph{$(s,\bbk)$-equal path space} is defined as
$\vec{P}(X_{\cf_{n,s}})_{\bzero}^{\bbk+\bone}$ where
$\cf_{n,s}=\{F \subset [n]: |F| \geq s\}$. This path space occurs when
every process $T_i$ calls upon a single resource $a$ \emph{of
  capacity} $s-1$ a number $k_i$ of times.

We use Formula (\ref{po2}) to obtain some information on the homology
of this space.  For $m\geq s$, let $\opi_{m,s}$ denote the poset of
nontrivial partitions of $[m]$ such that every non-singleton block has
cardinality at least $s$.  The homology of the order complex
$\Delta(\opi_{m,s})$ had been determined by Bj\"{o}rner and Welker
\cite{B-Welker95} and was further studied in \cite{B-Wachs96,PRW99}.
We will need the following result:
\begin{theorem}[Theorem 4.5 in \cite{B-Welker95}, Corollary 6.2 in \cite{B-Wachs96}]
  $\Delta(\opi_{m,s})$ has the homotopy type of a wedge of spheres.
  The $d$-th Betti number of $\Delta(\opi_{m,s})$ is nonzero iff
  $d=m-3-\ell(s-2)$ for some
  $1 \leq \ell \leq \lfloor\frac{n}{s}\rfloor$, and
\begin{equation}
\label{bequal}
\tilde{\beta}_{m-3-\ell(s-2)}(\Delta(\opi_{m,s}))=
\sum_{\substack{j_1+\cdots+j_{\ell} =m \\ j_i \geq s}} \binom{m-1}{j_1-1,j_2,\ldots,j_{\ell}}\prod_{i=0}^{\ell-1} \binom{j_i-1}{s-1}.
\end{equation}
\end{theorem}
\noindent
Note that $\Delta(M((\cf_{n,s})_{\prec F}))\cong \Delta(\opi_{|F|,s})$ for any
$F \in \cf_{n,s}$. Theorem \ref{poin1}(i) implies that
$H_*(\vec{P}(X_{\cf_{n,s}})_{\bzero}^{\bbk+\bone})$ is free.
Moreover, by Theorem \ref{poin1}(ii)
\begin{equation}
\begin{split}
\label{po4}
&f_{\KK}\left(\vec{P}(X_{\cf_{n,s}})_{\bzero}^{\bbk+\bone},t\right) \\
&=\sum_{0 \neq \bbm \in \Nat^{\cf_{n,s}}}
\bfk(\bbm)t^{c_{\cf_{n,s}}(\bbm)}
\prod_{F \in \cf_{n,s}} f_{\KK}\left(\Delta(M((\cf_{n,s})_{\prec F})),t^{-1}\right)^{\bbm(F)} \\
&= \sum_{0 \neq \bbm \in \Nat^{\cf_{n,s}}} \bfk(\bbm)t^{\sum_{F \in
    \cf_{n,s}}\bbm(F) (|F|-2)+1} \mu(\bbm,t)
\end{split}
\end{equation}
where
$$
\mu(\bbm,t)=\prod_{F \in \cf_{n,s}}
\left(\sum_{\ell=1}^{\lfloor\frac{|F|}{s}\rfloor}
  \tilde{\beta}_{|F|-3-\ell(s-2)}(\Delta(\opi_{|F|,s}))t^{-|F|+2+\ell(s-2)}
\right)^{\bbm(F)}.
$$
It follows that $t^{\alpha}$ appears in
$f_{\KK}\left(\vec{P}(X_{\cf_{n,s}})_{\bzero}^{\bbk+\bone},t\right)$
with nonzero coefficient only if $\alpha \equiv 1 (\bmod (s-2))$.
\begin{corollary}
\label{homnz}
$\tilde{H}_{\ell}(\vec{P}(X_{\cf_{n,s}})_{\bzero}^{\bbk+\bone};\Int)=0$
unless $\ell=m(s-2)$ for some $m>0$.
\end{corollary}

\subsection{The Connectivity of Path Spaces}
\label{st:conn}

The following result determines the
homological connectivity of $\ovp(X_{\cf})_{\bzero}^{\bbk+\bone}$.
\begin{proposition}
\label{prop:homconn}
Let $s(\cf)=\min_{F \in \cf} |F|$. Then
$$\min\{i: \thh_i(\ovp(X_{\cf})_{\bzero}^{\bbk+\bone};\Int) \neq 0\}=s(\cf)-2.$$
\end{proposition}
\noindent {\bf Proof:} Choose an $F \in \cf$ such that
$|F|=s(\cf)$. Then $\Delta(M(\cf)_{\prec F})$ is the empty complex
$\{\emptyset\}$ and therefore
$f_{\KK}(\Delta(M(\cf)_{\prec F}),t^{-1})=1$. Letting $\mb{m}(F')=1$
if $F=F'$ and zero otherwise, it follows from Theorem \ref{poin1}(ii)
that $t^{c_{\cf}(\mb{m})}=t^{|F|-1}=t^{s(\cf)-1}$ appears in
$f_{\KK}(\ovp(X_{\cf})_{\bzero}^{\bbk+\bone},t)$ with a positive
coefficient, and therefore
$\thh_{s(\cf)-2}(\ovp(X_{\cf})_{\bzero}^{\bbk+\bone};\KK) \neq 0$.
\ \\ \\
For the other direction, first note that for any $F \in \cf$
$$\dim \Delta(M(\cf)_{\prec F}) \leq |F|-s(\cf)-1.$$ Therefore for any
$F_1,\ldots,F_r \in \cf$
\begin{equation*}
\begin{split}
  \dim \overset{r}{\underset{i=1}{\varhexstar}}  \Delta(M(\cf)_{\prec F_i}) &\leq \sum_{i=1}^r (|F_i|-s(\cf)-1)+r-1 \\
  &=\sum_{i=1}^r |F_i|- r s(\cf)-1  \\
  &<\sum_{i=1}^r (|F_i|-2)-s(\cf)+2.
\end{split}
\end{equation*}
Thus $$\thh^j(S^{N-\sum_{i=1}^r (|F_i|-2)-1}  \ast
\overset{r}{\underset{i=1}{\varhexstar}}  \Delta(M(\cf)_{\prec F_i})=0$$
for all
\begin{equation*}
\begin{split}
j &\geq (N-\sum_{i=1}^r (|F_i|-2)-1)+(\sum_{i=1}^r (|F_i|-2)-s(\cf)+2)+1 \\
&=N-s(\cf)+2.
\end{split}
\end{equation*}
As $\ehat$ is a wedge of spaces of the form
$$S^{N-\sum_{i=1}^r (|F_i|-2)-1}  \ast
\overset{r}{\underset{i=1}{\varhexstar}} \Delta(M(\cf)_{\prec F_i})$$
where $F_1,\ldots,F_r \in \cf$, it follows that $\thh^j(\ehat;\Int)=0$
for all $j \geq N-s(\cf)+2$.  Finally, Alexander duality
$\thh_i(\ovp(X_{\cf})_{\bzero}^{\bbk+\bone};\Int) \cong
\thh^{N-i-1}(\ehat)$
implies that $\thh_i(\ovp(X_{\cf})_{\bzero}^{\bbk+\bone};\Int)=0$ for
all $i \leq s(\cf)-3$.  {\enp}

\noindent In fact, we establish the following stronger result:

\begin{proposition}
  Let $\mb{p}$ denote any directed path in
  $\vec{P}(X_{\cf})_{\bzero}^{\bbk+\bone}$. Then
  $\pi_i(\vec{P}(X_{\cf})_{\bzero}^{\bbk+\bone};\mb{p})=0$ for all
  $i \leq s(\cf)-3$.
\end{proposition}

\begin{proof}
  According to Proposition \ref{prop:homeq}, we may replace
  $\vec{P}(X_{\cf})_{\bzero}^{\bbk+\bone}$ with the homotopy
  equivalent space $D_{\cf}\subset\intd_{\bbk}$. Proposition
  \ref{prop:homconn} tells us that $D_{\cf}$ is connected; hence we
  can choose any base point $\mb{p}\in D_{\cf}$ in the
  following. Connectedness can also be concluded from the subsequent
  argument in the case $i=0$.

  Let $F:S^i\to D_{\cf}$ denote any continuous map. Its image $F(S^i)$
  is compact and has thus positive distance from the compact set
  $\overline{E_{\mc{F}}}\subset\Delta_{\bbk}$. $F$ admits a smooth
  approximation $\tilde{F}:S^i\to\intd_{\bbk}$ homotopic to $F$ and so
  close to $F$ that the image of the homotopy does not intersect
  $E_{\mc{F}}$. Extend $\tilde{F}$ to a smooth map
  $G:D^{i+1}\to\intd_{\bbk}$ by defining $G(0)=\mb{p}$ and by convex
  combination with $\tilde{F}$ on the boundary $S^i$. The image
  $G(D^{i+1})$ may intersect $E_{\cf}$.

  By multiple application of the transversality theorem (see
  e.g.~\cite[Theorem III.2.1]{Hirsch:76},\cite[Ch.~I.2]{AGV:85}), one
  can find a smooth approximation $H$ to $G$ that is transversal to
  all strata in $E_{\cf}$. Moreover, since the compact sets
  $G(D^{i+1})$ and $\partial\Delta_{\mb{k}}$ have a positive distance,
  we may assume that $H(D^{i+1})$ is contained in $\intd_{\bbk}$, as
  well. Each of the subspaces $G_{F,,\mb{j}}$ in the definition of
  $E_{\cf}$ has codimension $|F|-1$ in $\mb{R}^N$, and intersections
  have higher codimensions. In particular, if $i+1<|F|-1$, then
  $H(D^{i+1})\cap G_{F,,\mb{j}}=\emptyset$ by transversality. If
  $i+1<s(\cf)-1$, then $H(D^{i+1})\cap E_{\cf}=\emptyset$ and $H$
  establishes that $\tilde{F}$ and hence $F$ are nul-homotopic in
  $D_{\cf}$.
\end{proof}

\section{Concluding Remarks}
\label{s:remarks}
We conclude with a few remarks about possible extensions of the
results of this paper that we hope to deal with in future work. One
obvious challenge concerns finding maps from spheres, and more
generally products of spheres, into path space such that the images of
the fundamental classes may serve as generators for homology in the
appropriate dimensions, aiming at a generalization of \cite[Corollary
3.10]{RZ14} in the paper of Raussen and Ziemia\'{n}ski. This is work
in progress.

On the other hand, the situation we analysed is perhaps characterized
by more regularity than what is needed for the method to work. The
paper of Raussen and Ziemia\'{n}ski \cite{RZ14} calculates the
homology of the path space $\vec{P}(X)_{\mb{0}}^{\mb{k}+\mb{1}}$ with
$X=\Rea^n\setminus Y$ with $Y$ a \emph{subset} of $\Int^n$. It seems
likely that it is possible to extend our results to the following more
general situation (with $\cf$ an upward closed hypergraph on $[n]$ as
previously):

For $F\in\cf$ and $\alpha: F\to\Int$ a function, let
$Y_{\alpha}:=\{ (x_1,\dots ,x_n)|\; x_i=\alpha (i), i\in F\}$. For any
non-empty subset $\beta (F)\subset \Int^F$ let
$Y_{\beta (F)}:=\bigcup_{\alpha\in\beta (F)}Y_{\alpha}$. In the
present paper, we only considered $\beta (F)=\Int^F$.

Now we assume that for every $F\in\cf$ such a subset $\beta (F)$ has
been chosen. Coherence suggests either to make a choice only for
minimal elements of the family or to ask that $\beta (F_2)$ consists
of \emph{all extensions} of functions in $\beta (F_1)$ to $F_2$ in
case $F_1\subset F_2$. The set to be excluded is then the union of
hyperplanes $Y=\bigcup_{F\in\cf}Y_{\beta (F)}$. It seems likely that
one can determine the homology of
$\vec{P}(X)_{\mb{0}}^{\mb{k}+\mb{1}}$ with $X=\Rea^n\setminus Y$, as
well.

It is less obvious how to analyse topological properties of path
spaces associated to general PV spaces (cf.~Section \ref{s:intro}) via
arrangements -- those would no longer be given by restrictions of
\emph{linear} subspaces. Instead, one has to remove \emph{thickened}
subspace arrangements within products of simplices leading to pattern
spaces that are more difficult to analyse. For such thickened
arrangements, our method -- that makes essential use of the Wedge
Lemma -- is in general no longer applicable.

Since Ziemia\'{n}ski has shown \cite{Ziemianski:15} that every finite
simplicial complex can arise as a connected component of the path
space for some PV-space, one cannot expect a simple algorithmic
determination of the homology of such a path space in general.


\begin{thebibliography}{99}

\bibitem{AGV:85} V.I. Arnold and S.M. Gusein-Zade and A.N. Varchenko,
  Singularities of Differentiable Maps, Volume 1, Basel:
  Birkh\"{a}user, 1985.

\bibitem{B-Wachs96} A. Bj\"{o}rner and M. L. Wachs, Shellable nonpure
  complexes and posets. I, {\it Trans. Amer. Math. Soc.} {\bf
    348}(1996), 1299--1327.

\bibitem{B-Welker95} A. Bj\"{o}rner and V. Welker, The homology of
  "$k$-equal" manifolds and related partition lattices, {\it
    Adv. Math.} {\bf 110}(1995), 277--313.


\bibitem{Dijkstra:68}
E.W. Dijkstra, Co-operating sequential processes, Programming Languages
  (F.~Genuys, ed.), Academic Press, New York, 1968, 43--110.

\bibitem{FGHMR:15} L. Fajstrup, \'{E}. Goubault, E. Haucourt,
  S. Mimram and M. Raussen, Directed Algebraic Topology and
    Concurrency, Springer-Verlag, Berlin, 2016.

\bibitem{FGR:06} L. Fajstrup, \'{E}. Goubault, and M. Raussen,
  Algebraic {T}opology and {C}oncurrency, Theor. Comput. Sci
  \textbf{357}(2006), 241--278, Revised version of Aalborg University
  preprint, 1999.

\bibitem{Glabbeek:06} R. van Glabbeek, On the
    {E}xpressiveness of {H}igher {D}imensional {A}utomata,
  Theor.~Comput.~Sci. \textbf{368}(2006), 168--194.

\bibitem{GM88} M. Goresky and R. MacPherson, Stratified Morse theory,
  Springer-Verlag, Berlin, 1988.


\bibitem{Hirsch:76} M. Hirsch, Differential Topology, Springer-Verlag,
New York, 1976.

\bibitem{PRW99} I. Peeva, V. Reiner and V. Welker, Cohomology of real
  diagonal subspace arrangements via resolutions, {\it Compositio
    Math.} {\bf 117}(1999), 99--115.

\bibitem{Pratt:90}
V. Pratt, Modelling concurrency with geometry, Proc. of the 18th ACM
  Symposium on Principles of Programming Languages. (1991), 311--322.


\bibitem{Raussen:10}
M. Raussen, Simplicial models for trace spaces, Algebr. Geom. Topol.
  \textbf{10}(2010), 1683--1714.

\bibitem{Raussen:12a} M. Raussen, Simplicial models for
    trace spaces {I}{I}: General higher-dimensional automata,
  Algebr. Geom. Topol. \textbf{12}(2012), 1745--1765.

\bibitem{RZ14} M. Raussen and K. Ziemia\'{n}ski, Homology of spaces of
  directed paths on Euclidean cubical complexes, {\it J. Homotopy
    Relat. Struct.} {\textbf 9}(2014), 67--84.


\bibitem{Stanley73} R. P. Stanley, Acyclic orientations of graphs,
  {\it Discrete Math.} {\bf 5}(1973), 171--178.

\bibitem{wachs07} M. L. Wachs, Poset topology: tools and applications,
  in {\it Geometric combinatorics}, 497--615, IAS/Park City
  Math. Ser., 13, Amer. Math. Soc., Providence, RI, 2007.

\bibitem{walker88} J.W. Walker, Canonical homeomorphisms of posets,
  {\it European J. Combin.} {\bf 9}(1988), 97--107.

\bibitem{ZZ93} G. Ziegler and R. \v{Z}ivaljevi\'{c}, Homotopy types of
  subspace arrangements via diagrams of spaces, {\it Math. Ann.} {\bf
    295}(1993), 527--548.

\bibitem{Ziemianski:15} K. Ziemia\'{n}ski, On execution spaces of
  PV-programs, {\it Theoret. Comput. Sci.}, {\bf 619}(2016) 87-–98.
\end{thebibliography}
\end{document}